\documentclass[11pt]{article}
\usepackage{amssymb,amsthm}
\usepackage{amsmath}
\usepackage{amscd}
\usepackage{epsf,graphicx}
\usepackage[all]{xy}
\newtheorem{thm}{Theorem}
\newtheorem{lema}[thm]{Lemma}
\newtheorem{cor}[thm]{Corollary}
\newtheorem{rem}[thm]{Remark}

\newtheorem{defi}[thm]{Definition}
\newtheorem{exa}[thm]{Example}
\newtheorem{pro}[thm]{Problem}

\newcommand{\alg}{{\rm{alg}}}

\date{}

\begin{document}
\setlength{\baselineskip}{16pt}
\title{Quantum Product of Symmetric Functions}
\author{Rafael D\'\i az \ \   and  \ \ Eddy Pariguan}
\maketitle
\begin{abstract}
We provide an explicit description of the quantum product of multi-symmetric functions
using the elementary multi-symmetric functions introduced by Vaccarino.
\end{abstract}

\section{Introduction}

Fix a characteristic zero field $\mathbb{K}.$  The algebro-geometric duality  allow us to identify affine algebraic varieties with the $\mathbb{K}$-algebra of polynomial functions on it, and reciprocally, a finitely generated algebra without nilpotent elements may be identified with its spectrum,  provided with the Zarisky topology. Affine space $\mathbb{K}^n$ is thus identified with the algebra of polynomials in $n$-variables $\mathbb{K}[x_1,...,x_n]$.  Consider the action of the symmetric group $S_n$ on $\mathbb{K}^n$ by  permutation of vector entries. The quotient space $\mathbb{K}^n/S_n$  is the configuration space
of $n$ unlabeled points  with repetitions in $\mathbb{K}$.  Polynomial functions on $\mathbb{K}^n/S_n$ may be identified with
the algebra $\mathbb{K}[x_1,...,x_n]^{S_n}$ of $S_n$-invariant polynomials in $\mathbb{K}[x_1,...,x_n]$.
A remarkable classical fact  is that $\mathbb{K}^n/S_n$ is again a $n$-dimensional
affine space \cite{MacD}, indeed we have an isomorphism of algebras
$$\mathbb{K}[x_1,...,x_n]^{S_n} \  \simeq \ \mathbb{K}[e_1,...,e_n] $$
where, for $\alpha \in [n]= \{1,...,n\}$,  $e_{\alpha}$ is the elementary symmetric polynomial given by
$$e_{\alpha}(x_1,...,x_n)\ = \ \sum_{|a|=\alpha}\ x^{\alpha} \ = \ \sum_{a \subseteq [n],\ |a|=\alpha}\ \prod_{j \in a}x_j .$$
The elementary symmetric polynomials are determined by the identity
$$\prod_{i =1}^n(1+x_i t)\ = \  \sum_{\alpha=0}^ne_{\alpha}(x_1,...,x_n)t^{\alpha}.$$
Using characteristic functions one shows for $\alpha_1,...,\alpha_m \in [n]$ that $$e_{\alpha_1}\cdots e_{\alpha_m}\ = \ \sum_{a \in \mathbb{N}^n}c(\alpha_1,...,\alpha_m ,a)x^a$$ where
$c(\alpha_1,...,\alpha_m ,a)$  is the cardinality of the subset of matrices of format $n\times m$ with entries in
$\{0,1 \}$ such that
$$\sum_{j=1}^nA_{ij}=\alpha_i \ \ \mbox{for} \ \ i \in [m], \ \ \ \ \ \ \ \
\sum_{i=1}^mA_{ij}= a_j \ \ \mbox{for}\  \ j \in [n]. $$

A subtler  situation arises when one considers the configuration space $$(\mathbb{K}^d)^n/S_n$$
of $n$ unlabeled points with repetitions in $\mathbb{K}^d$,
for $d \geq 2.$ In this case $(\mathbb{K}^d)^n/S_n$ is not longer an affine space, instead it is an affine algebraic variety.
Polynomial functions on $(\mathbb{K}^d)^n/S_n$ are the so-called multi-symmetric functions, also known as vector
symmetric functions or MacMahon symmetric functions \cite{gel,MacD}, and coincides with the algebra of invariant polynomials $$\mathbb{K}[x_{11},..., x_{1d},......, x_{n1},..., x_{nd}]^{S_n},$$ which admits a presentation of the following form
$$\mathbb{K}[\ e_{\alpha} \ | \ |\alpha| \in [n]\ ]\ / \ \mathrm{I}_{n,d},$$ where
the elementary multi-symmetric functions $e_{\alpha}$, for  $\alpha=(\alpha_1,...,\alpha_d)\in \mathbb{N}^d$ a vector such that $|\alpha|= \alpha_1 + ... + \alpha_d \leq n,$ are defined by the identity
$$\prod_{i =1}^n(1+ x_{i1}t_1 + \cdots + x_{id}t_d)\ = \sum_{\alpha \in  \mathbb{N}^d, \ |\alpha| \leq n}e_{\alpha}(x_{11},..., x_{1d},......, x_{n1},..., x_{nd})t_1^{\alpha_1}...t_d^{\alpha_d}.$$
Explicitly, the multi-symmetric function $e_{\alpha}$ is given by
$$ e_{\alpha}(x_{11},..., x_{1d},......, x_{n1},..., x_{nd})\ = \ \sum_{a}x^a \ = \ \sum_{|a|=\alpha}\ \prod_{j \in [d]} \prod_{i \in a_j}x_{ij} ,$$
where in the middle term we regard $a \in \mathrm{M}_{n\times d}(\{0,1\})$ as a matrix such that
$$\sum_{i=1}^n a_{ij}=\alpha_j \ \ \ \mbox{for} \ \ j \in [d],  \ \ \ \ \ \sum_{j=1}^d a_{ij} \leq 1 \ \ \ \mbox{for} \ \ i \in [n]; \ \ \ \ \ \mbox{and}
\ \ \ \ \ x^a=\prod_{i=1}^n\prod_{j=1}^dx_{ij}^{a_{ij}}\ ;$$
and in the right hand side term we let
$a=(a_1,...,a_d)$ be a $d$-tuple of disjoint sets $a_j \subseteq [n]$ such that $$|a|\ = \ (|a_1|,...,|a_d|)\ = \ (\alpha_1,...,\alpha_d).$$

It is not difficult to check that any multi-symmetric function can be written (not uniquely) as a linear combination of products of elementary multi-symmetric functions. The non-uniqueness is controlled by the ideal $\mathrm{I}_{n,d}$.  For an explicit description of  $\mathrm{I}_{n,d}$ the reader
may consult Dalbec \cite{Dal} and Vaccarino \cite{Vac}. \\

One checks for $\ \alpha_1,...,\alpha_m \in \mathbb{N}^d \ $ that:
$$e_{\alpha_1}\cdots e_{\alpha_m}\ = \ \sum_{a \in \mathrm{M}_{n\times d}(\mathbb{N})}c(\alpha_1,...,\alpha_m ,a)x^a,$$
where $c(\alpha_1,...,\alpha_m ,a)$
counts the number of cubical matrices
$$A=(A_{ijl})\ \in \  \mathrm{Map}([m]\times [n] \times [d], \{0,1 \})$$ such that
$$\sum_{i=1}^mA_{ijl}= a_{jl} \ \ \mbox{for} \ \ j \in [n], \ l \in [d], \ \ \ \ \ \ \ \sum_{l=1}^dA_{ijl}\leq 1 \ \ \mbox{for} \ \ i \in [m], \  j \in [n],
 \ \ \ \ \ \mbox{and}$$ $$\sum_{j=1}^nA_{ijl}= (\alpha_i)_l \ \ \mbox{for} \ \ i \in [m], \ l \in [d] .$$

Recall that an algebra may be analyzed by describing it by generators and relations or
alternatively, as emphasized by Rota and his collaborators, by finding
a suitable basis such that the structural coefficients are positive integers with preferably a nice combinatorial interpretation. The second  approach for the case of multi-symmetric functions was undertaken by Vaccarino \cite{Vac}, and his results will be reviewed in Section \ref{4}.
The main goal of this work, see Section \ref{5}, is to generalize this combinatorial approach to multi-symmetric functions from the classical to the quantum setting.\\

Quantum mechanics, the century old leading small distances physical theory, is
still not quite fully understood by mathematicians. The transition from classical to quantum mechanics
has been particularly difficult to grasp. An appealing approach to this problem is to characterize the process of quantization as
a process of deformation of a commutative Poisson algebra into a non-commutative algebra \cite{bay}.
In this approach classical phase space is replaced by  quantum phase space, where an extra dimension parametrized by a formal variable $\hbar$ is added.\\

The classical phase space of a Lagrangian theory is naturally endowed with a closed two-form. In the non-degenerated case (i.e. in the symplectic case) this two-form can be inverted given rise to a Poisson brackect on the algebra of smooth functions on phase space. In a sense, the Poisson bracket may be regarded as a tangent vector in the space of deformations of the algebra of functions on phase space, i.e. as an infinitesimal deformation. That this infinitesimal deformation can be integrated into a formal deformation is a result due to Fedosov \cite{fed} for the symplectic case, and to Kontsevich \cite{Kon} for arbitrary Poisson manifolds.\\

Many Lagrangian physical theories are invariant under a continuous group of transformations; in that case the two-form on phase space is necessarily degenerated. Nevertheless, a Lagrangian theory might be invariant under a finite group and still retain its non-degenerated character. In the latter scenario  all the relevant constructions leading to the quantum algebra of functions on phase space are equivariant, and thus give rise to a quantum algebra of invariant functions  under the finite group. We follow this path along this work, being as explicit and calculative as possible. Our main aim is thus to provide foundations as well as practical tools for dealing with quantum symmetric functions.

\section{Multi-Symmetric Functions}\label{4}

In this section we introduce Vaccarino's multi-symmetric functions $e_{\alpha}(p)$ which are defined in analogy with the elementary multi-symmetric functions of the Introduction, a yet the definition is general enough to account for the symmetrization of arbitrary polynomial functions \cite{Vac}.\\

Fix $a,n,d\in \mathbb{N}^+.$ Let $\ y_1,...,y_d \ $ and $\ t_1,...,t_a \ $ be a pair of sets of commuting independent variables over $\mathbb{K}$.
For $\ \alpha = (\alpha_1, ... , \alpha_a)\in \mathbb{N}^a \ $  we set
$${\displaystyle|\alpha|=\sum_{i=1}^a \alpha_i}\ \ \ \ \ \ \ \mbox{and}\ \ \ \ \ \ \ {\displaystyle t^{\alpha}=\prod_{i=1}^{a}t_i^{\alpha_i}}.$$
For $\ q\in \mathbb{K}[y_1,...,y_d] \ $ and $\ i\in[n]\ $  we let $q(i)\ = \ q(x_{i1},...,x_{id})$ be the polynomial obtained by replacing each appearance of $y_j$ in $q$ by $x_{ij}$,  for $j\in[d]$. For example, for $n=2 \ $ and $\ q=y_1y_2y_3\in \mathbb{R}[y_1,y_2,y_3]$ we have that $$q(1)=x_{11}x_{12}x_{13} \ \ \ \ \ \mbox{and} \ \ \ \ \ q(2)=x_{21}x_{22}x_{23}.$$

\begin{defi}\label{vacdef}{\em
Consider $\alpha\in\mathbb{N}^a$ such that  $|\alpha|\leq n$, and
$p= (p_1,..., p_a)\in \mathbb{K}[y_1,...,y_d]^a$. The  multi-symmetric functions
$e_{\alpha}(p) \in \mathbb{K}[(\mathbb{K}^{d})^{n}]^{S_n}$ are determined by the identity:
\begin{equation}
\prod_{i=1}^{n}\left( 1+  p_{1}(i)t_{1} + \cdots +  p_{a}(i)t_{a} \right) \ = \ \sum_{|\alpha|\leq n}e_{\alpha}(p)t^{\alpha}.\nonumber
\end{equation}
}
\end{defi}

\begin{exa}{\em
For $n=3$ and $p=(y_1y_2,y_3y_4)$, we have that  $\ e_{(1,1)}(y_1y_2,y_3y_4)$ is equal to
$$x_{13}x_{14}x_{21}x_{22}\ +\ x_{21}x_{22}x_{33}x_{34}\ +\ x_{23}x_{24}x_{31}x_{32}+x_{11}x_{12}x_{33}x_{34}\ +\ x_{13}x_{14}x_{31}x_{32}\ +\ x_{11}x_{12}x_{23}x_{24},$$
$$\mbox{and} \ \ \ \ e_{(2,1)}(y_1y_2,y_3y_4)\ = \ x_{1 1}x_{12}x_{21}x_{22}x_{33}x_{34}\ +\ x_{11}x_{12}x_{23}x_{24}x_{31}x_{32}\ +\ x_{13}x_{14}x_{21}x_{22}x_{31}x_{32}.$$

}
\end{exa}

\begin{exa}{\em
For $p=(y_1,..., y_d)$ and $\alpha\in\mathbb{N}^d$ with  $|\alpha|\in [n]$, the multi-symmetric functions $e_{\alpha}(y_1,...,y_d)$ are
the elementary multi-symmetric functions defined in the Introduction.}
\end{exa}

Next couple of Lemmas follow directly from Definition  \ref{vacdef}.

\begin{lema}{\em Let $\alpha\in\mathbb{N}^a$ be such that  $|\alpha|\leq n$, and let $p= (p_1,..., p_a)\in \mathbb{K}[y_1,...,y_d]^a$. The  multi-symmetric function $e_{\alpha}(p)$ is given by the combinatorial identity
\begin{equation}
e_{\alpha}(p)\ = \ \sum_{c} \prod_{l=1}^a\prod_{i\in c_l}p_{l}(i),\nonumber
\end{equation}
where $c=(c_1,...,c_a)$ is a tuple  of disjoint subsets of $[n]$, with
$|c| =  (|c_1|,...,|c_a|)  =  \alpha.$
}
\end{lema}

Recall that the symmetrization map $\ \mathbb{K}[(\mathbb{K}^{d})^n] \ \longrightarrow \ \mathbb{K}[(\mathbb{K}^{d})^{n}]^{S_n} \ $ sends $f$ to
$$\sum_{\sigma \in S_n} f\circ \sigma,$$
where we regard $\sigma\in S_n$ as a map $\sigma: (\mathbb{K}^{d})^n \ \longrightarrow \ (\mathbb{K}^{d})^n.$

\begin{lema}\label{sim}
{\em Let $\alpha\in\mathbb{N}^a$ be such that  $|\alpha|\leq n$, and $p= (p_1,..., p_a)\in \mathbb{K}[y_1,...,y_d]^a$. The multi-symmetric function $e_{\alpha}(p)$ is the symmetrization of the polynomial
$$p_1(1)...p_1(\alpha_1)......p_i(\sum_{l=1}^{i-1}\alpha_l + 1) ... p_i(\sum_{l=1}^{i}\alpha_l) ...... p_a(\sum_{l=1}^{a-1}\alpha_l + 1) ... p_a(\sum_{l=1}^a\alpha_l).$$
}

\end{lema}

\begin{lema}\label{l6}
{\em Let $p= (p_1,..., p_a)\in \mathbb{K}[y_1,...,y_d]^a$ be expanded in monomials as
$$p_1 =    \sum_{j_1\in [k_1]}c_{1j_1}m_{1j_1}, \ ........,\ p_a =    \sum_{j_a\in [k_a]}c_{aj_a}m_{aj_a} .$$ Then
$$e_{\alpha}(p) \ = \  \ \sum_{\beta \in \mathbb{N}^{|k|}, \ r(\beta)= \alpha }e_{\beta}(m)c^{\beta},$$
where $m=(m_{11},...,m_{1k_1},......,m_{a1},...,m_{aj_a}), \ \ \ \ c = (c_{11},..., c_{1k_1},.......,c_{a1},..., c_{ak_a} ), \ \  \mbox{and for}$
$\beta=(\beta_{11},..., \beta_{1k_1},.......,\beta_{a1},..., \beta_{ak_a}) \in \mathbb{N}^{|k|} \ $ we set
$$r(\beta) \ = \ (\beta_{11}+ \cdots + \beta_{1k_1},.........,\beta_{a1}+ \cdots  +\beta_{ak_a})  .$$}
\end{lema}

\begin{proof} Let $\ |k| = k_1 + \cdots + k_a \ $ and
$\ ct= (c_{11}t_1,..., c_{1k_1}t_1,.......,c_{a1}t_1,..., c_{ak_a}t_a ).\ $ Then
$$ \sum_{\alpha \in \mathbb{N}^a, \ |\alpha|\leq n}e_{\alpha}(p)t^{\alpha} \ = \ \prod_{i=1}^{n}\left( 1+  \sum_{j_1\in [k_1]}c_{1j_1}m_{1j_1}(i)t_{1} + \cdots +  \sum_{j_a\in [k_a]}c_{aj_a}m_{aj_a}(i)t_{a} \right)$$ $$ \ \ \ \ \ \ \ \ \ \ \ \ \ \ \ \  \ = \ \sum_{\beta \in \mathbb{N}^{|k|}, \ |\beta|\leq n}e_{\beta}(m)(ct)^{\beta} \ = \ \sum_{\beta \in \mathbb{N}^{|k|}, \ |\beta|\leq n}e_{\beta}(m)c^{\beta}t^{r(\beta)}. $$
Thus we get that
$$e_{\alpha}(p) \ = \   \sum_{\beta \in \mathbb{N}^{|k|}, \ r(\beta)= \alpha }e_{\beta}(m)c^{\beta}.$$

\end{proof}

The following result due to Vaccarino \cite{Vac} provides an explicit formula for the product  of multi-symmetric functions. We include the proof since the same technique carries over to the more involved quantum case.

\begin{thm} \label{VACC}{\em
 Fix $a,b, n \in\mathbb{N}^+$,  $p\in \mathbb{K}[y_1,...,y_d]^a$, and $q\in \mathbb{K}[y_1,...,y_d]^b$.  Let $\alpha\in \mathbb{N}^a$ and  $\beta\in \mathbb{N}^b$ be such that $|\alpha|, |\beta|\leq n$. Then we have that:
$$e_{\alpha}(p)e_{\beta}(q)\ = \ \sum_{\gamma\in \mathrm{L}(\alpha,\beta,n)}e_{\gamma}(p,q,pq),\ \ \ \ \ \ \mbox{where:}$$
\begin{itemize}
\item  $(p,q,pq)\ = \ (p_1,..., p_a, q_1,..., q_b, p_1q_1,..., p_1q_b, ......, p_aq_1,..., p_aq_b).$

\item $\mathrm{L}(\alpha,\beta,n)$ is the set of matrices $\gamma\in \mathrm{Map}([0,a]\times[0,b], \mathbb{N})$  such that:
$$\gamma_{00}=0, \ \ {\displaystyle|\gamma|=\sum_{l=0}^a\sum_{r=0}^b\gamma_{lr}\leq n},   \ \ {\displaystyle\sum_{r=0}^b\gamma_{lr}=\alpha_l} \ \mathrm{for} \   l\in[a],   \ \  \mathrm{and}  \ \ {\displaystyle\sum_{l=0}^a\gamma_{lr}=\beta_r} \  \mathrm{for} \  r\in[b].$$
\end{itemize}
}
\end{thm}

\begin{proof}

Identify the matrix $\gamma$ with the vector
$$\gamma\ = \
(0, \gamma_{01},...,\gamma_{0b},\gamma_{10},...,\gamma_{a0}, \gamma_{20},...,\gamma_{2b}, ......, \gamma_{a0},...,\gamma_{ab}).$$
We have that $$\sum_{ |\alpha|,|\beta|\leq n}e_{\alpha}(p)e_{\beta}(q)t^{\alpha}s^{\beta}$$ is equal to
\begin{eqnarray*}
&=&
\left(\sum_{|\alpha|\leq n}e_{\alpha}(p)t^{\alpha}\right)\left(\sum_{|\beta|\leq n}e_{\beta}(q)s^{\beta}\right)\\
\mbox{} &=& \prod_{i=1}^{n}\left(1\ +\ \sum_{l=1}^a p_l(i) t_l\right) \prod_{i=1}^{n}\left(1\ + \ \sum_{r=1}^b q_r (i)s_r\right)\\
\mbox{} &=& \prod_{i=1}^{n}\left(1\ + \ \sum_{l=1}^a p_l(i) t_l\right)\left(1\ + \ \sum_{r=1}^b q_r (i)s_r\right)\\
\mbox{} &=& \prod_{i=1}^{n}\left(1\ + \ \sum_{l=1}^a p_l(i) t_l \ + \ \sum_{r=1}^b q_r (i)s_r\ + \ \sum_{l=1}^a \sum_{r=1}^b p_l(i) q_r (i) t_l s_r\right)\\
\mbox{} &=& \prod_{i=1}^{n}\left(1\ + \ \sum_{l=1}^ap_l(i)w_{l0}\ + \ \sum_{r=1}^bq_r(j)w_{0r}\ + \ \sum_{l=1}^a \sum_{r=1}^b p_l(i)q_r (i)w_{lr}\right)\\
\mbox{} &=& \sum_{\gamma}e_{\gamma}(p,q,pq)w^{\gamma},
\end{eqnarray*}
where for $\gamma \in \mathrm{L}(\alpha,\beta,n)$ we set
 $$w^{\gamma}\ = \ \prod\limits_{l=0}^{a}\prod\limits_{r=0}^{b}w_{lr}^{\gamma_{lr}}\ = \ \prod\limits_{l=0}^{a}\prod\limits_{r=0}^{b}(t_ls_r)^{\gamma_{lr}}\ = \
 \prod\limits_{l=0}^{a}\prod\limits_{r=0}^{b}t_l^{\gamma_{lr}}s_r^{\gamma_{lr}},$$
using the conventions $$t_0=s_0=1,\  \  \ \  \ \ \ \ \  w_{lr}=t_ls_r  \ \ \ \ \mbox{for}\ \ \ l,r \geq 0.$$
For $w^{\gamma}$ to be equal to $t^{\alpha}s^{\beta}$ we must have that
 $$w^{\gamma} \ = \ \prod\limits_{l=0}^{a}\prod\limits_{r=0}^{b}t_l^{\gamma_{lr}}s_r^{\gamma_{lr}} \ = \
 \left( \prod\limits_{l=1}^{a}t_l^{\sum\limits_{r=0}^{b}\gamma_{lr}}\right)\left( \prod\limits_{r=1}^{b}s_r^{\sum\limits_{l=0}^{a}\gamma_{lr}}\right) \ = \
 \left( \prod\limits_{l=1}^{a}t_l^{\alpha_l}\right)\left( \prod\limits_{k=1}^{b}s_r^{\beta_r}\right).$$
 Thus we conclude that  $${\displaystyle\sum_{r=0}^b\gamma_{lr}=\alpha_l}\  \ \mbox{for}\ \ l\in[a],\ \ \ \ \ \ \ \ \ \ \mbox{and}\ \ \ \ \ \ \ \ \ \ {\displaystyle\sum_{l=0}^a\gamma_{lr}=\beta_r}\ \ \mbox{for}\ \ r\in[b].$$
\end{proof}
\noindent Graphically, a matrix $\gamma\in \mathrm{L}(\alpha,\beta,n)$ is represented as
$$\begin{array}{cccccccc}
0 &\gamma_{01}&\gamma_{02}&\gamma_{03}&\cdots&\gamma_{0b}& \mbox{ }&\mbox{ }\\
\gamma_{10}&\gamma_{11}&\gamma_{12}&\gamma_{13}&\cdots&\gamma_{1b}& \longrightarrow&\alpha_1\\
\gamma_{20}&\gamma_{21}&\gamma_{22}&\gamma_{23}&\cdots&\gamma_{2b}& \longrightarrow&\alpha_2\\
\mbox{ } &\vdots&\vdots&\vdots&\mbox{ } &\vdots & \mbox{ }&\vdots\\
\mbox{ } &\vdots&\vdots&\vdots&\mbox{ } &\vdots & \mbox{ }&\vdots\\
\gamma_{a0}&\gamma_{a1}&\gamma_{a2}&\gamma_{a3}&\cdots&\gamma_{ab}& \longrightarrow&\alpha_a\\
\mbox{ } &\downarrow&\downarrow&\downarrow& \mbox{ } &\downarrow & \mbox{ }&\mbox{ }\\
\mbox{ } &\beta_1&\beta_2&\beta_3&\cdots &\beta_b & \mbox{ }&\mbox{ }
\end{array}$$
where the horizontal and vertical arrows represent, respectively,  row and column sums.

\begin{exa}{\em  For  $n=3$, $\ p=(y_1y_2,y_1)\ $ and  $\ q=(y_1y_2,y_3)\ $ we have that:
$$e_{(1,1)}(y_1y_2,y_1)e_{(2,1)}(y_1y_2,y_3)\ = \ \sum_{\gamma}e_{\gamma}(y_1y_2,y_1,y_1y_2,y_3,y_1^2y_2^2,y_1y_2y_3,y_1^2y_2,y_1y_3),$$
where $\gamma=(\gamma_{10},\gamma_{20},\gamma_{01},\gamma_{02},\gamma_{11},\gamma_{12},\gamma_{21},\gamma_{22})\in\mathbb{N}^8$ is such that $\ |\gamma|\leq 3\ $ and $$\gamma_{10}+\gamma_{11}+\gamma_{12}=1, \ \ \ \ \ \
\gamma_{20}+\gamma_{21}+\gamma_{22}=1, \ \ \ \ \ \
\gamma_{01}+\gamma_{11}+\gamma_{21}=2, \ \ \ \ \ \
\gamma_{02}+\gamma_{12}+\gamma_{22}=1.$$
Looking at the solutions in $\mathbb{N}$ of  the  system of  linear equations above we obtain that:
$$e_{(1,1)}(y_1y_2,y_1)e_{(2,1)}(y_1y_2,y_3)=$$
$$e_{(1,1,1)}(y_3,y_1^2y_2^2,y_1^2y_2)\ +\ e_{(1,1,1)}(y_1y_2,y_1y_2y_3,y_1^2y_2)\ + \ e_{(1,1,1)}(y_1y_2,y_1^2y_2^2,y_1y_3).$$}
\end{exa}

\begin{exa}
{\em  For $n=4$, $\ p=(y_1^2y_2,y_2^3y_3,y_1y_2y_3)\ $ and $\ q=(y_1^3y_2^2y_3,y_1^2y_3,y_2y_3)\ $ we have that:
$$e_{(1,1,1)}(y_1^2y_2,y_2^3y_3,y_1y_2y_3)e_{(1,2,1)}(y_1^3y_2^2y_3,y_1^2y_3,y_2y_3)\ =$$
$$\sum_{\gamma}e_{\gamma}(y_1^2y_2,y_2^3y_3,y_1y_2y_3,y_1^3y_2^2y_3,y_1^2y_3,y_2y_3,y_1^5y_2^3,y_1^4y_2y_3,$$
$$y_1^2y_2^2y_3,y_1^3y_2^5y_3^2,y_2^4y_3^2, y_1y_2^4y_3,y_1^4y_2^3y_3^2,y_1^3y_2y_3^2,y_1^2y_2^2y_3),$$
where $\gamma=
(\gamma_{10},\gamma_{20},\gamma_{30},\gamma_{01},\gamma_{02},\gamma_{03},\gamma_{11},\gamma_{12},
\gamma_{13},\gamma_{21},\gamma_{22},\gamma_{23},\gamma_{31},\gamma_{32},\gamma_{33})\in\mathbb{N}^{15}$ is such that $|\gamma|\leq 4\ $ and
$$ \gamma_{10}+\gamma_{11}+\gamma_{12}+\gamma_{13}=1, \ \ \ \
\gamma_{20}+\gamma_{21}+\gamma_{22}+\gamma_{23}=1, \ \ \ \
\gamma_{30}+\gamma_{31}+\gamma_{32}+\gamma_{33}=1,$$
$$\gamma_{01}+\gamma_{11}+\gamma_{21}+\gamma_{31}=1, \ \ \ \
\gamma_{02}+\gamma_{12}+\gamma_{22}+\gamma_{32}=2, \ \ \ \
\gamma_{03}+\gamma_{13}+\gamma_{23}+\gamma_{33}=1.$$
Looking at the solutions in $\mathbb{N}$ of  the system of  linear equations above we obtain that:
$$e_{(1,1,1)}(y_1^2y_2,y_2^3y_3,y_1y_2y_3)e_{(1,2,1)}(y_1^3y_2^2y_3,y_1^2y_3,y_2y_3)\ =$$

$$\begin{array}{lccc}
e_{(1,1,1,1)}(y_2y_3,y_1^4y_2y_3,y_2^4y_3^2,y_1y_2y_3)& +& e_{(1,1,1,1)}(y_2y_3,y_1^4y_2y_3,y_1^3y_2^5y_3^2,y_1^3y_2y_3^2) &+\\
e_{(1,1,1,1)}(y_2y_3,y_1^5y_2^3,y_2^4y_3^2,y_1^3y_2y_3^2)& + &e_{(1,1,1,1)}(y_1^2y_3,y_1^2y_2^2y_3,y_2^4y_3^2,y_1^4y_2^3y_3^2)& +\\
e_{(1,1,1,1)}(y_1^2y_3,y_1^2y_2^2y_3,y_1^3y_2^5y_3^2,y_1^3y_2y_3^2)& +&  e_{(1,1,1,1)}(y_1^2y_3,y_1^4y_2y_3,y_1y_2^4y_3,y_1^4y_2^3y_3^2)& +\\
e_{(1,1,1,1)}(y_1^2y_3,y_1^4y_2y_3,y_1^3y_2^5y_3^2,y_1^2y_2^2y_3)& +&  e_{(1,1,1,1)}(y_1^2y_3,y_1^5y_2^3,y_1y_2^4y_3,y_1^3y_2y_3^2) &+\\
e_{(1,1,1,1)}(y_1^2y_3,y_1^5y_2^3,y_2^4y_3^2,y_1^2y_2^2y_3)& + & e_{(1,1,1,1)}(y_1^3y_2^2y_3,y_1^2y_2^2y_3,y_2^4y_3^2,y_1^3y_2y_3^2)& +\\
e_{(1,1,1,1)}(y_1^3y_2^2y_3,y_1^4y_2y_3,y_1y_2^3y_3^2,y_1^2y_2^2y_3)& + & e_{(1,1,1,1)}(y_1^3y_2^2y_3,y_1^4y_2y_3,y_2^3y_3^2,y_1^2y_2^2y_3).&\mbox{}
\end{array}$$}
\end{exa}

\section{Review of Deformation Quantization}\label{sec2}

In this section we review a few needed notions on deformation quantization.  We assume the reader to be somewhat familiar with Kontsevich's work $\cite{Kon}$, although that level of generality is not necessary to understand the applications to the quantization of canonical phase space. A  Poisson bracket \cite{PS, Izu} on a smooth manifold $M$ is a $\mathbb{R}$-bilinear antisymmetric map $$\{\ , \ \}:C^{\infty}(M)\times C^{\infty}(M) \ \longrightarrow \ C^{\infty}(M),$$ where
$C^{\infty}(M)$ is the space of real-valued smooth functions on $M$, and for $f,g,h\in C^{\infty}(M)$ the following identities hold:
$$ \{f,gh\}\ = \ \{f,g\}h\ +\ g\{f,h\} \ \ \ \ \  \mbox{and} \ \ \ \ \  \{f, \{g,h\}\}\ = \ \{\{f,g\},h\}\ +\ \{g, \{f,h\}\} .$$
A manifold equipped with a Poisson bracket is called a Poisson manifold.  The Poisson bracket $\{\ , \ \}$ is determined by an
antisymmetric bilinear form $\alpha$ on $T^{\ast}M$, i.e. by the Poisson bi-vector $\alpha\in \bigwedge^2T M$ given in
local coordinates $(x_1,x_2,\cdots,x_d)$ on $M$ by $$\alpha_{ij}=\{x_i,x_j\}.$$ The bi-vector $\alpha$ determines the Poisson bracket as follows
$$\{f,g\}\ = \ \alpha (df,dg) \ = \ \sum_{i,j\in [d]}\alpha_{ij}\frac{\partial f}{\partial x_i}\frac{\partial g}{\partial x_j},
\ \ \ \ \ \mbox{for} \ \  f,g\in C^{\infty}(M).$$
If the Poisson bi-vector $\alpha_{ij}$ is non-degenerated (i.e. $\mathrm{det}(\alpha_{i,j}) \neq 0$)  the Poisson manifold $M$ is called symplectic.

\begin{exa}{\em
The space $\mathbb{R}^{2d}$  is a symplectic Poisson manifold with Poisson bracket given in the linear coordinates $(x_1,...,x_d,y_1,...,y_d)$ by
$$\{f,g\} \ = \
\sum_{i=1}^d\left(\frac{\partial f}{\partial x_i} \frac{\partial g}{\partial y_i}\ - \ \frac{\partial f}{\partial y_i} \frac{\partial g}{\partial x_i}\right),
\ \ \ \ \ \mbox{for} \ \ f,g \in C^{\infty}(\mathbb{R}^{2d}).$$
Equivalently, the Poisson bracket $\{\ , \ \}$ on $C^{\infty}(\mathbb{R}^{2d})$  is determined by the identities:
$$\{x_i,x_j\}=0,\ \ \ \ \ \ \{y_i,y_j\}=0,\ \ \ \ \ \ \mbox{and} \ \ \ \ \ \ \{x_i,y_j\}=\delta_{ij}, \ \ \ \ \mbox{for} \ \ \  i,j\in[d].$$ This example is the so-called canonical phase space with $n$ degrees of freedom.}
\end{exa}

\begin{exa} {\em Let $({\mathfrak{g}},[ \mbox{ },\mbox{ } ])$ be a Lie algebra over $\mathbb{R}$ of dimension $d$.  The dual vector space ${\mathfrak{g}}^{\ast}$ is a Poisson manifold with Poisson bracket given on $f,g\in C^{\infty}({\mathfrak{g}}^{\ast})$  by
$$\{f,g\}(\alpha)\ = \ \langle \alpha,[d_{\alpha}f,d_{\alpha}g]\rangle,$$
where  $\alpha\in{\mathfrak{g}}^{\ast},$  and the differentials $d_{\alpha}f$ and $d_{\alpha}g$ are regarded as elements of ${\mathfrak{g}}$ via the identifications $T^{\ast}_\alpha{\mathfrak{g}}^{\ast}={\mathfrak{g}}^{\ast \ast}={\mathfrak{g}}$. Choose a linear basis $\ e_1,...,e_d \ $ for $\ {\mathfrak{g}}$.
The structural coefficients $c_{ij}^k$ of ${\mathfrak{g}}$ are given, for $i,j,k \in[d]$, by
$$[e_i,e_j]\ = \ \sum_{k=1}^d c_{ij}^k e_k.$$ Let $(x_1,...,x_d)$ be the linear system of coordinates on ${\mathfrak{g}}^{\ast}$ relative to the basis $e_1,...,e_d$ of $\mathfrak{g}$. The Poisson bracket is determined by continuity and the identities
$$\{x_i,x_j\}\ = \ \sum_{k=1}^d  c_{ij}^k x_k.$$}
\end{exa}

\

A formal deformation, or deformation quantization, of a Poisson manifold $M$
is an associative product,
called the star product,  $$\star: C^{\infty}(M)
[[\hbar]]\otimes_{\mathbb{R}[[\hbar]]} C^{\infty}(M)[[\hbar]]\ \longrightarrow \ C^{\infty}(M)[[\hbar]]$$
defined on the space $C^{\infty}(M)[[\hbar]]$ of formal power series in $\hbar$ with coefficients in $C^{\infty}(M)$ such that the following conditions hold for  $f,g\in C^{\infty}(M) $:

\begin{itemize}

\item $f\star g\ = \ \displaystyle{\sum_{n=0}^{\infty}B_n(f,g)\hbar^{n}}, \ $ where the maps
$$B_n(\ , \ ): C^{\infty}(M)\times  C^{\infty}(M)\ \longrightarrow \ C^{\infty}(M)$$ are bi-differential operators.

\item{$f\star g \ = \ fg\ + \ \frac{1}{2}\{f,g\}\hbar\ + \ O(\hbar^{2})$, where $O(\hbar^{2})$
stand for terms of order $2$ and higher in  the variable $\hbar$.}
\end{itemize}

Kontsevich in $\cite{Kon}$ constructed a  $\star$-product
for any finite dimensional Poisson manifold. For linear Poisson manifolds the Kontsevich $\star$-product goes as follows. Fix a Poisson manifold $(\mathbb{R}^{d},\alpha)$,
the Kontsevich $\star$-product is given on $f,g\in C^{\infty}(M) $ by
$$f\star g\ = \ \sum_{n=0}^{\infty} B_n(f,g) \frac{\hbar^{n}}{n!} \ = \
\sum_{n=0}^{\infty}\left( \sum_{\Gamma \in \mathbb{G}_n}
\omega_\Gamma B_{\Gamma}(f,g) \right)\frac{\hbar^{n}}{n!}, \ \ \ \mbox{where:}$$
\begin{itemize}
\item{$\mathbb{G}_n$ is a collection of graphs, called admissible graphs, each  with $2n$ edges.  }

\item{For each graph $\Gamma\in \mathbb{G}_n$, the constant $\omega_{\Gamma}\in\mathbb{R}$ is independent of $d$ and $\alpha$, and its computed trough an integral in an appropriated configuration space.
    }

\item{$ B_{\Gamma}(\ , \ ): C^{\infty}(\mathbb{R}^{d})\times  C^{\infty}(\mathbb{R}^{d})\ \longrightarrow \ C^{\infty}(\mathbb{R}^{d})$ is a bi-differential operator associated to the graph $\Gamma\in \mathbb{G}_n$ and the Poisson bi-vector $\alpha$. The definition of the operators
    $B_{\Gamma}(\ , \ )$ is quite explicit and fairly combinatorial in nature.}

\end{itemize}

\begin{rem}{\em Kontsevich himself have highlighted the fact that explicitly computing the integrals defining the constants $\omega_{\Gamma}$ is a daunting task currently beyond reach. One can however use the symbols $\omega_{\Gamma}$ as variables, and they will defined a deformation quantization (with an extended ring of constants) as soon as this variables satisfy a certain system of quadratic equations \cite{DP}.
}
\end{rem}

We are going to use the Kontsevich $\star$-product in a slightly modified form.  $$\mbox{Let}\ \  \mathbb{G}= \bigsqcup_{n=0}^{\infty}\mathbb{G}_n  \ \  \mbox{and for} \ \ \Gamma \in \mathbb{G}, \   \mbox{set}
\ \overline{\Gamma}=n \  \ \mbox{if and only if} \ \ \Gamma \in \mathbb{G}_n.$$
With this notation the Kontsevich $\star$-product is given on functions  $f,g \in C^{\infty}(\mathbb{R}^{d})$ by
$$f\star g \ = \ \sum_{\Gamma \in \mathbb{G}}\frac{\omega_{\Gamma}}{\overline{\Gamma}!} B_{\Gamma}(f,g)\hbar^{\overline{\Gamma}} .$$

\begin{rem}{\em The Kontsevich $\star$-product is defined over $\mathbb{R} $ since  $c_{\Gamma} \in \mathbb{R}$. If $\alpha$ is a regular Poisson bi-vector, i.e. the entries $\alpha_{ij}$ of the Poisson bi-vector are polynomial functions,  then the $\star$-product on
$C^{\infty}(\mathbb{R}^{d})$ restricts to a well-defined $\star$-product on the space  $\mathbb{R}[x_1,...,x_d]$ of polynomial functions on $\mathbb{R}^{d}$. We are interested in the quantization of symmetric polynomial functions, thus we assume that $\alpha$ is regular Poisson bi-vector and work with quantum algebra $(\mathbb{R}[x_1,...,x_d], \star)$.
}
\end{rem}

\section{Quantum Symmetric Functions}\label{sec3}

Let $S_n$ be the symmetric group on $n$ letters. For each subgroup $K\subseteq S_n$, consider the Polya functor $\mathrm{P}_K:\mathbb{R}\mbox{-} \alg \ \longrightarrow \ \mathbb{R}\mbox{-}\alg$  from the category of
associative $\mathbb{R}$-algebras to itself, defined on objects as
follows \cite{DP}. Let $A$ be a $\mathbb{R}$-algebra, the underlying vector space of $\mathrm{P}_K A$ is given by
$$\mathrm{P}_K A \ = \ (A^{\otimes n})_K\ = \ A^{\otimes n}\ / \ \langle \ a_1\otimes \dots \otimes
a_n \ - \ a_{\sigma 1}\otimes \dots \otimes a_{\sigma n}\  \ \ |\ \ \  a_i\in A, \ \sigma \in
K \ \rangle .$$
Elements of $\mathrm{P}_K A$ are written as $\ \overline{a_1\otimes \cdots \otimes a_n}. \ $
For $a_{ij}\in A$, the following identity  determines the product on $\mathrm{P}_K A$:
$$|K|^{m-1}\prod_{i=1}^{m}\left(\overline{\bigotimes_{j=1}^{n}
a_{ij}}\right)  \ \ = \ \sum_{\sigma \in \{1\} \times
K^{m-1}}\overline{\bigotimes_{j=1}^{n}\left( \prod_{i=1}^{m} a_{i
\sigma^{-1}_i(j)}\right)}.$$
The Polya functor $\mathrm{P}_K$ is also known as the co-invariants functor. The invariants functor
$$\begin{array}{cccc}
\mathrm{I}^K:& \mathbb{R}\mbox{-}
\alg &\longrightarrow &\mathbb{R}\mbox{-}
\alg\\
\mbox{}& A&\longrightarrow & (A^{\otimes
n})^{K}\\
\end{array}
$$ is given on objects by $$\ (A^{\otimes n})^{K}\ = \ \{\ a\in A^{\otimes n}\ \ | \ \  g  a=a\ \ \ \mbox{for} \ \ \ g\in K \ \}. \  $$
The product on $(A^{\otimes n})^{K}$ comes from  the inclusion $(A^{\otimes n})^K\subset A^{\otimes n}$. \\

The functors $\ \mathrm{I}^K \ $ and $\ \mathrm{P}_K \ $ are naturally isomorphic to each other \cite{DP}.\\

Suppose a finite group $K$ acts on a Poisson manifold $M$, and that the induced action of $K$ on  $(C^{\infty}(M)[[\hbar]] ,\star)$ is by algebra automorphisms,  then we define the algebra of  quantum $K$-symmetric functions on $M$ as
$$\  (C^{\infty}(M)[[\hbar]] ,\star)_K \ \simeq \ (C^{\infty}(M)[[\hbar]],\star)^{K}.$$

Let $(\mathbb{R}^d, \alpha)$ be a regular Poisson manifold. The Cartesian product of Poisson manifolds is naturally endowed with the structure of a Poisson manifold, thus we get a regular Poisson manifold structure on $(\mathbb{R}^d)^n$. We use the following coordinates on the  $n$-fold Cartesian product of $\mathbb{R}^d$ with itself: $$(\mathbb{R}^d)^n\ = \ \{\ (x_1,...,x_n)\ \ | \ \ x_i=(x_{i1},...,x_{id})\in \mathbb{R}^d ,\ \ x_{ij}\in\mathbb{R},\ \ (i,j)\in[n]\times[d]\ \}.$$
The ring of regular functions on $(\mathbb{R}^d)^n$ is the ring of polynomials  on $dn$ commutative variables:
$$\mathbb{R}[(\mathbb{R}^d)^n] \ = \ \mathbb{R}[x_{11},...,x_{1d},......,x_{n1},...,x_{nd}].$$
Consider another set of   commutative variables $y_1,...,y_d$. Recall from Section \ref{4} that
for $f\in \mathbb{R}[y_{1},...,y_{d}]$ and $i\in[n]$ we set $f(i)=f(x_{i1},...,x_{id})\in \mathbb{R}[x_{i1},...,x_{id}] \subseteq \mathbb{R}[(\mathbb{R}^d)^n].$\\

The Poisson bracket on $(\mathbb{R}^d)^n$ is determined by the following identities
$$\{x_{ki},x_{lj}\}\ = \ \delta_{kl}\alpha_{ij}(k), \ \ \ \ \ \mbox{for} \ \  i,j \in [d] \ \ \mbox{and} \ \ k,l \in [n]$$ where the coordinates $\alpha_{ij}$ of the Poisson bivector $\alpha=\sum \alpha_{ij}\partial_i\wedge \partial_j$ are regarded as polynomials in $\mathbb{R}[y_{1},...,y_{n}]$.  The Poisson bracket on $(\mathbb{R}^d)^n$ is $S_n$-invariant, indeed for $\sigma\in S_n$ we have that:
$$\{x_{\sigma k i},x_{\sigma l j}\}\ = \ \delta_{\sigma k,\sigma l}\alpha_{ij}(\sigma k)\ = \ \sigma(\delta_{kl}\alpha_{ij}(k)).$$

Next results  \cite{DP} provide a natural construction of groups acting as algebra automorphisms on
the algebras $(\mathbb{R}[\mathbb{R}^{d}][[\hbar]],\star)$.

\begin{thm} {\em
Let $(\mathbb{R}^{d},\{ \ , \ \} )$ be a regular Poisson manifold, and $K$ be a subgroup of $S_d$  such that the Poisson bracket
$\{\mbox{ },\mbox{ }\}$ is $K$-equivariant. Then the action of $K$ on
$\ (\mathbb{R}[\mathbb{R}^{d}][[\hbar]],\star) \ $ is by algebra automorphisms.}
\end{thm}

\begin{cor} {\em
Let $(\mathbb{R}^{d},\{ \ , \ \} )$ be a regular Poisson manifold and consider a subgroup $K\subseteq S_n$.
Then  $K$ acts  by algebra automorphisms on  $(\mathbb{R}[(\mathbb{R}^d)^{n}][[\hbar]],\star)$.}
\end{cor}

\begin{defi}\label{qsf}{\em Let $(\mathbb{R}^{d},\{ \ , \ \} )$ be a regular Poisson manifold.
The algebra of quantum
symmetric functions on $(\mathbb{R}^{d})^{n}$ is given by
$$\  (\mathbb{R}[(\mathbb{R}^d)^{n}][[\hbar]],\star)_{S_{n}}\  \simeq\
(\mathbb{R}[(\mathbb{R}^d)^{n}][[\hbar]],\star)^{S_{n}}.$$

}
\end{defi}

\begin{exa} {\em Consider $\mathbb{R}^{2d}$ with its canonical symplectic Poisson structure, then $(\mathbb{R}^{2d})^n$ is also a symplectic Poisson manifold. Choose coordinates on $(\mathbb{R}^{2d})^n$  as follows:
$$(\mathbb{R}^{2d})^n\ = \ \left\{\ (x_1,y_1,...,x_n,y_n)\ \mid \ x_i=(x_{i1},...,x_{id}), \ \  y_i=(y_{i1},...,y_{id}),\ \ x_{ij}, \ y_{ij}\in\mathbb{R}\ \right\}.$$ The $S_n$-invariant Poisson bracket on $(\mathbb{R}^{2d})^n$ is given for $i,j \in [d]$ and $k,l \in [n]$ by
$$\{x_{ki},x_{lj}\}=0,\ \ \ \  \ \ \ \{y_{ki},y_{lj}\}=0 \ \ \ \ \ \  \ \mbox{and} \ \ \ \ \  \ \ \{x_{ki},y_{lj}\}=\delta_{kl}\delta_{ij}.$$}
\end{exa}

\begin{exa} {\em Let $\mathfrak{g}$ be a $d$-dimensional Lie algebra over $\mathbb{R}$ and $\mathfrak{g}^{\ast}$ its dual vector space. Then $\mathfrak{g}^{\ast}$ is a Poisson manifold, and therefore $(\mathfrak{g}^\ast)^n$ is also a Poisson manifold. The $S_n$-invariant Poisson bracket on $(\mathfrak{g}^\ast)^n$ is given, for $i,j \in [d]$ and $k,l \in [n],$ by
$$\{x_{ki},x_{lj}\}\ = \ \delta_{kl}\sum_{m=0}^d c^m_{ij}x_{km}$$ where $c^m_{ij}$ are the structural coefficients of $\mathfrak{g}$.}
\end{exa}

Specializing Definition \ref{qsf} we obtain the following natural notions.
The algebra of quantum symmetric functions on $ (\mathbb{R}^{2d})^{n}$ is given by
$$(\mathbb{R}[(\mathbb{R}^{2d})^{n}][[\hbar]],\star)_{S_{n}}\ \ \simeq\ \
(\mathbb{R}[(\mathbb{R}^{2d})^{n}][[\hbar]],\star)^{S_{n}}.$$
More generally, the algebra of quantum symmetric functions on $(\mathfrak{g}^\ast)^n$ is given by
$$(\mathbb{R}[(\mathfrak{g}^\ast)^n][[\hbar]],\star)_{S_{n}}\ \ \simeq\ \
(\mathbb{R}[(\mathfrak{g}^\ast)^n][[\hbar]],\star)^{S_{n}}.$$

\section{$\star$-Product of Multi-Symmetric Functions}\label{5}

We are ready to state and proof the main result of this work which extends Theorem  \ref{VACC} from the classical to the quantum case: we
provide an explicit formula for the $\star$-product of multi-symmetric functions. \\

Recall that the $\star$-product can be expanded as a formal power series en $\hbar$ as:
$$f\star g\ = \ \sum_{n=0}^{\infty} B_n(f,g) \frac{\hbar^{n}}{n!} .$$

\begin{thm}\label{prin}{\em Let $(\mathbb{R}^{d},\{ \ , \ \} )$ be a regular Poisson manifold and $(\mathbb{R}[(\mathbb{R}^{d})^{n}][[\hbar]],\star)^{S_{n}}$
be the algebra of quantum symmetric functions on $(\mathbb{R}^{d})^{n}.$ Fix $a,b, n \in\mathbb{N}^+$, $p\in R[y_1,...,y_d]^a$, and $q\in R[y_1,...,y_d]^b$.  Let $\alpha\in \mathbb{N}^a$ and  $\beta\in \mathbb{N}^b$ be such that $|\alpha|, |\beta|\leq n$. The $\star$-product of $e_{\alpha}(p)$ and $e_{\beta}(q)$ is given by:
$$e_{\alpha}(p)\star e_{\beta}(q)\ = \ \sum_{m=0}^\infty\left(
\sum_{\gamma\in \mathrm{Q}(\alpha,\beta,n,m)}e_{\gamma}(B(p,q))\right)\hbar^m, \ \ \ \ \ \mbox{where:}$$
\begin{itemize}

\item  $B(p,q)=(p,q,...,B_k(p,q),...)\ $ and
$$B_k(p,q)\ = \ (B_k(p_1,q_1),..., B_k(p_1,q_b),......, B_k(p_a,q_1), ..., B_k(p_a,q_b)).$$

\item $\mathrm{Q}(\alpha,\beta,n,m)$ is the subset of   $\mathrm{Map}([0,a]\times[0,b] \times \mathbb{N}, \mathbb{N})$ consisting of cubical matrices $$\gamma: [0,a]\times[0,b] \times \mathbb{N} \ \longrightarrow \ \mathbb{N}\ \ \ \ \ \mbox{such that:}$$

\begin{itemize}
\item  $\gamma_{00k}=0 \ $ for $\ k\geq 0; \ \ \ $ if either $\ l=0\ $ or $\ r=0, \ $ then $\ \gamma_{lrk}=0\ $ for $\ k\geq 1$.

\item $|\gamma| = \sum\limits_{l=0}^{a}\sum\limits_{r=0}^{b}\sum\limits_{k=0}^{\infty}\gamma_{lrk} \leq n, \ \ \ \ \ \mbox{and}
\ \ \ \ \ \sum\limits_{l=0}^{a}\sum\limits_{r=0}^{b}\sum\limits_{k=0}^{\infty}k\gamma_{lrk}=m.$

\item $\sum\limits_{r=0}^{b}\sum\limits_{k=0}^{\infty}\gamma_{lrk}=\alpha_l \ \ \  \mbox{for}\ \ l\in [a],
\ \ \ \ \ \mbox{and} \ \ \ \ \ \sum\limits_{l=0}^{a}\sum\limits_{k=0}^{\infty}\gamma_{lrk}=\beta_r \ \ \ \mbox{for}\ \  j\in [b].$

\end{itemize}

\end{itemize}}
\end{thm}

\begin{proof}
We have that
$$\sum_{|\alpha|, |\beta| \leq n}\sum_{m=0}^{\infty}B_m(e_{\alpha}(p), e_{\beta}(q))t^{\alpha}s^{\beta}\hbar^m \ = \
\sum_{|\alpha|, |\beta| \leq n}\left(e_{\alpha}(p)\star e_{\beta}(q)\right)t^{\alpha}s^{\beta} \ = $$
$$\left(\sum_{|\alpha| \leq n}e_{\alpha}(p)t^{\alpha}\right)\star \left(\sum_{|\beta| \leq n}e_{\beta}(q)s^{\beta}\right)\ = \
\prod_{i=1}^{n}\left(1\ + \ \sum_{l=1}^ap_l(i)t_l\right)\star  \prod_{i=1}^{n}\left(1\ + \ \sum_{r=1}^bq_r(i)s_r\right)\ = $$
$$\prod_{i=1}^{n}\left(1\ +\ \sum_{l=1}^ap_l(i)t_l\right)\star \left(1\ +\ \sum_{r=1}^bq_r(i)s_r\right)\ = $$
$$\prod_{i=1}^{n}\left(1\ +\ \sum_{l=1}^ap_l(i)t_l\ + \ \sum_{r=1}^bq_r(i)s_r \ + \ \sum_{l=1}^a \sum_{r=1}^b p_l(i)\star q_r(i) t_l s_r \right)\ = $$
$$\prod_{i=1}^{n}\left(1\ +\ \sum_{l=1}^ap_l(i)t_l \ + \ \sum_{r=1}^bq_r(i)s_r \ + \
\sum_{l=1}^a \sum_{r=1}^b \sum_{k=0}^{\infty} B_k(p_l(i), q_r(i))  t_l s_r \hbar^k \right)\ = $$
$$\prod_{i=1}^{n}\left(1\ +\ \sum_{l=1}^a p_l(i)w_{l00}\ + \ \sum_{r=1}^bq_r(i)w_{0r0}\ + \ \sum_{l=1}^a \sum_{r=1}^b \sum_{k=0}^{\infty}  B_k( p_l(i), q_r(i))w_{lrk}\right)\ = $$
$$\sum_{\gamma\in \mathrm{Q}(\alpha,\beta,n,m)}e_{\gamma}(B(p,q))w^{\gamma}, \ \ \ \ \ \  \ \  \mbox{where:}$$

$$w^{\gamma}\ = \  \prod\limits_{l=0}^{a}\prod\limits_{r=0}^{b}\prod\limits_{k=0}^{\infty}w_{lrk}^{\gamma_{lrk}}\
 = \ \prod\limits_{l=0}^{a}\prod\limits_{r=0}^{b}\prod\limits_{k=0}^{\infty}(t_ls_r\hbar^k)^{\gamma_{lrk}}\ = \ \prod\limits_{l=0}^{a}\prod\limits_{r=0}^{b}\prod\limits_{k=0}^{\infty}t_l^{\gamma_{lrk}}s_r^{\gamma_{lrk}}\hbar^{k\gamma_{lrk}},$$ and we are using the conventions
 $$t_0=s_0=1,\ \ \ \ \ \ \  \ w_{ruk}=t_rs_u\hbar^k\ \ \ \ \ \mbox{for}\ \ \ \ r,u,m\geq 0.$$
 For $w^{\gamma}$ to be equal to $t^{\alpha}s^{\beta}\hbar^m$ we must have

$$\left(\prod\limits_{l=0}^{a}t_l^{\alpha_l}\right)\left(\prod\limits_{r=0}^{b}s_r^{\beta_r}\right)\hbar^m \ = \ \prod\limits_{l=0}^{a}\prod\limits_{r=0}^{b}\prod\limits_{k=0}^{\infty}t_l^{\gamma_{lrk}}s_r^{\gamma_{lrk}}\hbar^{k\gamma_{lrk}} \ = $$

$$ \left ( \prod\limits_{l=0}^{a}\prod\limits_{r=0}^{b}\prod\limits_{k=0}^{\infty}t_l^{\gamma_{lrk}}\right)
 \left(   \prod\limits_{l=0}^{a}\prod\limits_{r=0}^{b}\prod\limits_{k=0}^{\infty}s_r^{\gamma_{lrk}}\right)
 \left( \prod\limits_{l=0}^{a}\prod\limits_{r=0}^{b}\prod\limits_{k=0}^{\infty}\hbar^{k\gamma_{lrk}}\right) = $$

$$\left( \prod\limits_{l=1}^{a}t_l^{\sum\limits_{r=0}^{b}\sum\limits_{k=0}^{\infty}\gamma_{lrk}}\right)
\left( \prod\limits_{r=1}^{b}s_r^{\sum\limits_{l=0}^{a}\sum\limits_{k=0}^{\infty}\gamma_{lrk}}\right)
\hbar^{\sum\limits_{l=0}^{a}\sum\limits_{r=0}^{b}\sum\limits_{k=0}^{\infty}k\gamma_{lrk}}.$$
and thus we conclude that
$$\sum\limits_{r=0}^{b}\sum\limits_{k=0}^{\infty}\gamma_{lrk}=\alpha_l \ \ \mbox{for}\ l\in [a], \ \ \  \ \ \
\sum\limits_{l=0}^{a}\sum\limits_{k=0}^{\infty}\gamma_{lrk}=\beta_r \ \  \mbox{for}\  j\in [b], \ \ \ \ \ \
\sum\limits_{l=0}^{a}\sum\limits_{r=0}^{b}\sum\limits_{k=0}^{\infty}k\gamma_{lrk}=m.$$

\end{proof}

\begin{cor}\em{ With the assumptions of Theorem \ref{prin}, the Poisson bracket of the multi-symmetric functions $e_{\alpha}(p)$ and  $e_{\beta}(q)$ is given by
$$\{e_{\alpha}(p), e_{\beta}(q)\}\ = \ 2\sum_{\gamma\in \mathrm{Q}(\alpha,\beta,n,1)}e_{\gamma}(B(p,q)).$$}
\end{cor}
\begin{proof}Follows from Theorem \ref{prin}  and the identity
$$\{e_{\alpha}(p), e_{\beta}(q)\}\ = \ 2 \frac{\partial}{\partial \hbar} (e_{\alpha}(p)\star e_{\alpha}(q))\mid_{\hbar=0}.$$
\end{proof}

For our next result we regard  $(\mathbb{R}[(\mathbb{R}^{d})^{n}][[\hbar]],\star)^{S_{n}}$ as a topological algebra with topology induced by the inclusion $$(\mathbb{R}[(\mathbb{R}^{d})^{n}][[\hbar]],\star)^{S_{n}} \ \subseteq \
(\mathbb{R}[(\mathbb{R}^{d})^{n}][[\hbar]],\star),$$ where a fundamental system of neighborhoods of $0 \in
\mathbb{R}[(\mathbb{R}^{d})^{n}][[\hbar]]$ is given by the decreasing family of sub-algebras
$$\mathbb{R}[(\mathbb{R}^{d})^{n}][[\hbar]] \ \supseteq \  \hbar \mathbb{R}[(\mathbb{R}^{d})^{n}][[\hbar]] \ \supseteq ......... \ \supseteq \ \hbar^n\mathbb{R}[(\mathbb{R}^{d})^{n}][[\hbar]] \ \supseteq ..........$$

Recall from the introduction that the elementary multi-symmetric functions $e_k$, for  $k\in \mathbb{N}^d$ with $|k|\leq n,$ are defined by the identity
$$\prod_{i =1}^n(1+ x_{i1}t_1 + \cdots + x_{id}t_d)\ = \sum_{k \in  \mathbb{N}^d, \ |k| \leq n}e_kt^{k}.$$

Similarly, the homogeneous multi-symmetric functions $h_k$, for  $k=(k_1,...,k_d)\in \mathbb{N}^d,$ are defined by the identity
$$\prod_{i =1}^n\frac{1}{1- x_{i1}t_1 - \cdots - x_{id}t_d} \  = \ \sum_{k \in  \mathbb{N}^d}h_kt^{k}.$$

Let $\mathcal{M}_d$ be the set of (non-trivial) monomials in the variables $y_1,...,y_d$. The power sum symmetric function $e_1(m)$ is given, for $m \in \mathcal{M}_d$, by
$$e_1(m)\ = \ m(1) \ + \ \cdots \ + \ m(d).$$

\begin{thm}
{\em
The elementary  multi-symmetric functions $e_k$ for $|k|\leq n,$ the homogeneous multi-symmetric functions $h_k$ for $|k|\leq n,$ and the power sum  multi-symmetric functions $e_1(m)$ with $m \in \mathcal{M}_d$ a monomial of degree less than or equal to $n$,  together with $\hbar$ generate, respectively, the topological algebra $(\mathbb{R}[(\mathbb{R}^{d})^{n}][[\hbar]],\star)^{S_{n}}$.
}
\end{thm}

\begin{proof}  It is  known \cite{Bri,Dal,Fles,MacD, Vac} that each of the aforementioned sets of multi-symmetric functions generate the algebra of classical multi-symmetric functions $\mathbb{R}[(\mathbb{R}^{d})^{n}]^{S_{n}}$. \\

To go the quantum case the same argument is applied in each case, so
we only consider the elementary symmetric functions. Take $\ f \in (\mathbb{R}[(\mathbb{R}^{d})^{n}][[\hbar]],\star)^{S_{n}} \ $ and expand it as formal power series
$$f \ = \  \sum_{k=0}^{\infty}f_k \hbar^k \ \ \ \ \ \mbox{with} \ \ \ \ \ f_k \in \mathbb{R}[(\mathbb{R}^{d})^{n}]^{S_{n}}.$$
We can write $f_0$ as a linear combination of a product of elementary symmetric functions. For simplicity assume that $\ f_0 = e_{k_1}\cdots e_{k_m}, \ $ then $$f \ - \ e_{k_1}\star \cdots \star e_{k_m}\ \in \ O(\hbar).$$ Assume next that $\ (f \ - \ e_{k_1}\star \cdots \star e_{k_m})_1 \ $ can be written as $\ e_{l_1} \cdots  e_{l_r}, \ $ then
$$f \ - \ e_{k_1}\star \cdots \star e_{k_m} \ - \ e_{l_1}\star \cdots \star e_{l_r}\hbar
\ \in \ O(\hbar^2).$$
Proceeding by induction we see that $f$ can be written as a formal power series in $\hbar$ with coefficients
equal to the sum of the $\star$-product of elementary multi-symmetric functions.
\end{proof}

Choose a variable $t_m$ for each $m \in \mathcal{M}_d$, the set of monomials in the variables $y_1,...,y_d$,  and set
$$\sum_{|\alpha| \leq n}e_{\alpha}t^{\alpha} \ = \ \prod_{i=1}^n\left(1 \ + \sum_{m \in \mathcal{M}_d}m(i)t_m \right),$$
where $\alpha:\mathcal{M}_d \ \longrightarrow \ \mathbb{N}$ of finite support, and
$t^{\alpha}=\prod_{m \in \mathcal{M}_d}t_m^{\alpha(m)}. $

\begin{thm}
{\em
The set $\{e_{\alpha}\hbar^k \ | \ |\alpha| \leq n, \ k \geq 0\}$ is a topological basis for the topological algebra $(\mathbb{R}[(\mathbb{R}^{d})^{n}][[\hbar]],\star)^{S_{n}}.$ The product of basic elements is given by Lemma \ref{l6} and Theorem \ref{prin}.
}
\end{thm}

\begin{proof}
It is well-known that the symmetrization of monomials yields a basis for $\mathbb{R}[(\mathbb{R}^{d})^{n}]^{S_{n}}$, thus from Lemma \ref{sim} we see that the set
$\{e_{\alpha} \ | \ |\alpha| \leq n\}$ is a basis for $\mathbb{R}[(\mathbb{R}^{d})^{n}]^{S_{n}}$ as well.
Thus forming the products $e_{\alpha}\hbar^k$ we obtain a topological basis for
$(\mathbb{R}[(\mathbb{R}^{d})^{n}][[\hbar]],\star)^{S_{n}}.$
\end{proof}

Next result describes the product of multi-symmetric functions using the Kontsevich's $\star$-product. In this case one can give a more precise formula for the computation of the quantum higher corrections, i.e. the coefficients that accompany the higher order powers in $\hbar$.

\begin{thm}
\label{kont}{\em Let $(\mathbb{R}^{d},\{ \ , \ \} )$ be a regular Poisson manifold and $(\mathbb{R}[(\mathbb{R}^{d})^{n}][[\hbar]],\star)^{S_{n}}$
be the algebra of  quantum symmetric functions on $(\mathbb{R}^{d})^{n}$ with the Kontsevich $\star$-product. Fix $a,b, n \in\mathbb{N}^+$, $p\in R[y_1,...,y_d]^a$, and $q\in R[y_1,...,y_d]^b$.  Let $\alpha\in \mathbb{N}^a$ and  $\beta\in \mathbb{N}^b$ be such that $|\alpha|, |\beta|\leq n$. The $\star$-product of $e_{\alpha}(p)$ and $e_{\beta}(q)$ is given by:
$$e_{\alpha}(p)\star e_{\beta}(q)\ = \ \sum_{m=0}^\infty\left(
\sum_{\gamma\in \mathrm{K}(\alpha,\beta,n,m)}e_{\gamma}(B(p,q))\right)\hbar^m, \ \ \ \ \ \mbox{where:}$$
\begin{itemize}

\item  $B(p,q)=(p,q,......,\frac{ \omega_{\Gamma}}{\overline{\Gamma}!}B_{\Gamma}(p,q),......)\ $ and
$$B_{\Gamma}(p,q)\ = \ \left(B_{\Gamma}(p_1,q_1),..., B_{\Gamma}(p_1,q_b),..., B_{\Gamma}(p_a,q_1), ..., B_{\Gamma}(p_a,q_b)\right).$$ The polynomial $B_{\Gamma}(p_i,q_j)$ results of applying Kontsevich's bi-differential operator $B_{\Gamma}$ to the pair $(p_i, q_j)$.

\item $\mathrm{K}(\alpha,\beta,n,m)$ is the subset of   $\mathrm{Map}([0,a]\times[0,b] \times \mathrm{G}, \mathbb{N})$ consisting of maps $$\gamma: [0,a]\times[0,b] \times \mathrm{G} \ \longrightarrow \ \mathbb{N}
    \ \ \ \ \ \mbox{such that:}$$

\begin{itemize}
\item  $\gamma_{00\Gamma}=0;$  \ \ \  if either $\ l=0\ $ or $\ r=0, \ $ then $\ \gamma_{lr\Gamma}=0\ $ for $\ \overline{\Gamma} \geq 1$.

\item $|\gamma| = \sum\limits_{l=0}^{a}\sum\limits_{r=0}^{b}\sum\limits_{\Gamma \in \mathrm{G}}\gamma_{lr\Gamma} \leq n, \ \ \ \ \ \mbox{and}
\ \ \ \ \ \sum\limits_{l=0}^{a}\sum\limits_{r=0}^{b}\sum\limits_{\Gamma \in \mathrm{G}}\overline{\Gamma}\gamma_{lr\Gamma}=m.$

\item $\sum\limits_{r=0}^{b}\sum\limits_{\Gamma \in \mathrm{G}}\gamma_{lr\Gamma}=\alpha_l \ \ \  \mbox{for}\ \ l\in [a],
\ \ \ \ \ \mbox{and} \ \ \ \ \ \sum\limits_{l=0}^{a}\sum\limits_{\Gamma \in \mathrm{G}}\gamma_{lr\Gamma}=\beta_r \ \ \ \mbox{for}\ \  j\in [b].$

\end{itemize}

\end{itemize}}
\end{thm}

\begin{proof} We have that
$$\sum_{|\alpha|, |\beta| \leq n}\sum_{m=0}^{\infty}B_m(e_{\alpha}(p), e_{\beta}(q))t^{\alpha}s^{\beta}\hbar^m \ =
\ \sum_{|\alpha|, |\beta| \leq n}\left(e_{\alpha}(p)\star e_{\beta}(q)\right)t^{\alpha}s^{\beta}\ = $$
$$\left(\sum_{|\alpha| \leq n}e_{\alpha}(p)t^{\alpha}\right)\star \left(\sum_{|\beta| \leq n}e_{\beta}(q)s^{\beta}\right)\ = \
\prod_{i=1}^{n}\left(1\ +\ \sum_{l=1}^ap_l(i)t_l\right)\star \left(1\ +\ \sum_{r=1}^bq_r(i)s_r\right)\ = $$
$$\prod_{i=1}^{n}\left(1\ +\ \sum_{l=1}^ap_l(i)t_l\ + \ \sum_{r=1}^bq_r(i)s_r \ + \ \sum_{l=1}^a \sum_{r=1}^b p_l(i)\star q_r(i) t_l s_r \right)\ = $$
$$\prod_{i=1}^{n}\left(1\ +\ \sum_{l=1}^ap_l(i)t_l \ + \ \sum_{r=1}^bq_r(i)s_r \ + \
\sum_{l=1}^a \sum_{r=1}^b  \sum_{\Gamma \in \mathrm{G}} \frac{ \omega_{\Gamma} }{\overline{\Gamma}!}B_{\Gamma}(p_l(i),q_r(i))t_ls_r\hbar^{\overline{\Gamma}} \right)\ = $$
$$\prod_{i=1}^{n}\left(1\ +\ \sum_{l=1}^a p_l(i)w_{l0\emptyset}\ + \ \sum_{r=1}^bq_r(i)w_{0r\emptyset}\ + \ \sum_{l=1}^a \sum_{r=1}^b \sum_{\Gamma \in \mathrm{G}}\frac{ \omega_{\Gamma}}{\overline{\Gamma}!}B_{\Gamma}(p_l(i),q_r(i))w_{rk\Gamma}\right) \ = $$
$$\sum_{\gamma\in \mathrm{K}(\alpha,\beta,n,m)}e_{\gamma}(B(p,q))w^{\gamma}, \ \ \ \ \ \  \ \  \mbox{where:}$$

$$w^{\gamma}\ = \ \prod\limits_{l=0}^{a}\prod\limits_{r=0}^{b}\prod\limits_{\Gamma \in \mathrm{G}}w_{lr\Gamma}^{\gamma_{lr\Gamma}}\ = \ \prod\limits_{l=0}^{a}\prod\limits_{r=0}^{b}\prod\limits_{\Gamma \in \mathrm{G}}^{\infty}t_l^{\gamma_{lrk}}s_r^{\gamma_{lrk}}\hbar^{\overline{\Gamma}\gamma_{lr\Gamma}},$$ by convention $\emptyset$ stands for the unique graph in $\mathrm{G}$ with no edges (representing the classical product), and
 $$t_0=s_0=1,\ \ \ \ \ \ \  \ w_{ru\Gamma}=t_rs_u\hbar^{\overline{\Gamma}}\ \ \ \ \ \mbox{for}\ \ \ \ r,u\geq 0, \ \Gamma \in \mathbb{G}.$$
 For $t^{\alpha}s^{\beta}\hbar^m = w^{\gamma}$ we must have

$$\left(\prod\limits_{l=0}^{a}t_l^{\alpha_l}\right)\left(\prod\limits_{r=0}^{b}s_r^{\beta_r}\right)\hbar^m \ = \ \prod\limits_{l=0}^{a}\prod\limits_{r=0}^{b}\prod\limits_{\Gamma \in \mathrm{G}}t_l^{\gamma_{lrk}}s_r^{\gamma_{lrk}}\hbar^{\overline{\Gamma}\gamma_{lr\Gamma}} \ = $$
$$\left( \prod\limits_{l=1}^{a}t_l^{\sum\limits_{r=0}^{b}\sum\limits_{\Gamma \in \mathrm{G}}\gamma_{lr\Gamma}}\right)
\left( \prod\limits_{r=1}^{b}s_r^{\sum\limits_{l=0}^{a}\sum\limits_{\Gamma \in \mathrm{G}}\gamma_{lr\Gamma}}\right)
\hbar^{\sum\limits_{l=0}^{a}\sum\limits_{r=0}^{b}\sum\limits_{\Gamma \in \mathrm{G}}\overline{\Gamma}\gamma_{lr\Gamma}}.$$
Thus we conclude that:
$$\sum\limits_{r=0}^{b}\sum\limits_{\Gamma \in \mathrm{G}}\gamma_{lr\Gamma}=\alpha_l \ \ \mbox{for}\ l\in [a], \ \ \  \ \
\sum\limits_{l=0}^{a}\sum\limits_{\Gamma \in \mathrm{G}}\gamma_{lr\Gamma}=\beta_r \ \  \mbox{for}\  j\in [b], \ \ \ \ \
\sum\limits_{l=0}^{a}\sum\limits_{r=0}^{b}\sum\limits_{\Gamma \in \mathrm{G}}\overline{\Gamma}\gamma_{lr\Gamma}=m.$$

\end{proof}

\section{Symmetric Powers of the Weyl Algebras}

In this section we study the case of two dimensional canonical phase space, i.e. the symplectic manifold
$\mathbb{R}^{2}$ with the canonical Poisson bracket given as follows:
$$\{f,g \} \ = \ \frac{\partial f}{\partial x}\frac{\partial g}{\partial y} -
\frac{\partial f}{\partial y}\frac{\partial g}{\partial x}.$$

\begin{defi}{\em
The Weyl algebra is defined  by generators and relations by
$$W \ = \ \mathbb{R} \langle x,y \rangle[[\hbar]]/
\langle yx-xy-\hbar \rangle.$$ }
\end{defi}

The deformation quantization of $(\mathbb{R}^{2}, \{ \ , \ \})$ is well-known to be given by the Moyal product \cite{mo}.
Moreover, one has the following result.

\begin{thm}{\em
The Weyl algebra is isomorphic to the deformation quantization of polynomial functions
on $\mathbb{R}^{2}$ with the canonical  Poisson structure.}
\end{thm}

Our goal in this section is to study the deformation quantization of $(\mathbb{R}^2)^n/S_n$, which can be identified the algebra of quantum symmetric functions $$(\mathbb{R}[x_1,...,x_n,y_1,...,y_n][[\hbar]], \star )^{S_n},$$  or equivalently,
with the symmetric powers of the Weyl algebra $$(W^{\otimes n})/S_n.$$

One shows by induction \cite{DP} that the following identity
holds in the Weyl algebra:

$$(x^{c}y^{d})\star (x^{f}y^{g}) \ = \ \sum_{k=0}^{\mathrm{min}} B_k(x^{c}y^{d},x^{f}y^{g}) \hbar^k \ = \
 \sum_{k=0}^{\mathrm{min}} {d \choose k} (f)_k x^{c+f-k}y^{d+g-k} \hbar^k,$$

$$\mbox{ where} \ \ \ \ \ \min=\min(d,f) \ \ \ \ \ \ \ \mbox{and} \ \ \ \ \ \ \ (f)_k=f(f-1)(f-2)(f-k+1).$$

\begin{thm}\label{voy}
\em{ Consider $\mathbb{R}^{2}$ with its canonical Poisson structure.  Fix $a,b\in\mathbb{N}^+$, and let $$(x^{c_1}y^{d_1},...,x^{c_a}y^{d_a})\in \mathbb{R}[x,y]^a \ \ \ \ \ \ \mbox{and} \ \ \ \ \ \
(x^{f_1}y^{g_1},...,x^{f_b}y^{g_b})\in \mathbb{R}[x,y]^b.$$ For $\alpha \in \mathbb{N}^a$ and $\beta \in \mathbb{N}^b$ the following identity holds:
$$e_{\alpha}(x^{c_1}y^{d_1},...,x^{c_a}y^{d_a})\star e_{\beta}(x^{f_1}y^{g_1},...,x^{f_b}y^{g_b})\ =$$
$$\sum_{m=0}^{\infty}\left( \sum_{\gamma\in \mathrm{Q}(\alpha,\beta,n,m)}e_{\gamma}
(B(x^{c_1}y^{d_1},...,x^{c_a}y^{d_a},x^{f_1}y^{g_1},...,x^{f_b}y^{g_b}))\right)\hbar^m,$$ where:
$$B(x^{c_1}y^{d_1},...,x^{c_a},y^{d_a},x^{f_1}y^{g_1},...,x^{f_b}y^{g_b})\ =$$
$$(x^{c_1}y^{d_1},...,x^{c_a},y^{d_a},x^{f_1}y^{g_1},..., x^{f_b}y^{g_b},.....,{d_r \choose k} (f_r)_kx^{c_l+f_r-k}y^{d_l+g_r-k} ,.....).$$
}
\end{thm}
\begin{proof} We have that:

$$\sum_{|\alpha|, |\beta| \leq n}\sum_{m=0}^{\infty}B_m(e_{\alpha}(x^{c_1}y^{d_1},...,x^{c_a}y^{d_a}), e_{\beta}(x^{f_1}y^{g_1},...,x^{f_b}y^{g_b}))t^{\alpha}s^{\beta}\hbar^m \ =$$

$$\sum_{|\alpha|, |\beta| \leq n}\left(e_{\alpha}(x^{c_1}y^{d_1},...,x^{c_a}y^{d_a})\star e_{\beta}(x^{f_1}y^{g_1},...,x^{f_b}y^{g_b})\right)t^{\alpha}s^{\beta} \ =$$

$$\left(\sum_{|\alpha|\leq n}e_{\alpha}(x^{c_1}y^{d_1},...,x^{c_a}y^{d_a})t^{\alpha}\right)\star \left(\sum_{|\beta|\leq n}e_{\beta}(x^{f_1}y^{g_1},...,x^{f_b}y^{g_b})s^{\beta}\right)\ =$$

$$\prod_{i=1}^n\left( 1\ +\ \sum_{l=1}^ax_i^{c_l}y_i^{d_l}t_l\right)\star\left( 1\ + \ \sum_{r=1}^bx_i^{f_r}y_i^{g_r}s_r\right)\ =$$

$$\prod_{i=1}^{n}\left(1\ + \ \sum_{l=1}^ax_i^{c_l}y_i^{d_l}t_l+\sum_{r=1}^bx_i^{f_r}y_i^{g_r}s_r \ + \
\sum_{l=1}^a \sum_{r=1}^b x_i^{c_l}y_i^{d_l} \star x_i^{f_r}y_i^{g_r}t_l s_r \right) \ =$$

$$\prod_{i=1}^{n}\left(1\ + \ \sum_{l=1}^ax_i^{c_l}y_i^{d_l}t_l+\sum_{r=1}^bx_i^{f_r}y_i^{g_r}s_r\ + \
\sum_{l=1}^a \sum_{r=1}^b  \sum_{k=0}^{\mathrm{min}} {d_r \choose k} (f_r)_kx_i^{c_l+f_r-k}y_i^{d_l+g_r-k} t_l s_r \hbar^k \right)\ =$$

$$\prod_{i=1}^{n}\left(1\ + \ \sum_{l=1}^ax_i^{c_l}y_i^{d_l}w_{l00}\ + \ \sum_{r=1}^bx_i^{f_r}y_i^{g_r}w_{0r0}\ + \
\sum_{l=1}^a \sum_{r=1}^b  \sum_{k=0}^{\mathrm{min}} {d_r \choose k} (f_r)_kx_i^{c_l+f_r-k}y_i^{d_l+g_r-k}w_{lrk} \right)\ =$$

$$\sum_{\gamma\in \mathrm{Q}(\alpha,\beta,n,m)}e_{\gamma}
(B(x^{c_1}y^{d_1},...,x^{c_a},y^{d_a},x^{f_1}y^{g_1},...,x^{f_b}y^{g_b}))w^{\gamma},$$
where $\min=\min\{d_l,f_r\}\ $ and $\ w_{lrk} =  t_l s_r \hbar^k$.
\end{proof}

\begin{cor}{\em With the assumptions of Theorem \ref{voy}, the Poisson bracket of the multi-symmetric functions $e_{\alpha}(p)$ and  $e_{\beta}(q)$ is given by
$$\{e_{\alpha}(p), e_{\beta}(q)\}\ = \ 2\sum_{\gamma\in \mathrm{Q}(\alpha,\beta,n,1)}e_{\gamma}
(B(x^{c_1}y^{d_1},...,x^{c_a},y^{d_a},x^{f_1}y^{g_1},...,x^{f_b}y^{g_b})).$$}
\end{cor}

\begin{exa}{\em Let $p=y, \ q=x \ \in \mathbb{R}[x,y]$. We have that:
$$e_{\alpha}(y)\star e_{\beta}(x) =\sum_{m\geq 0}\sum_{\gamma\in \mathrm{Q}(\alpha,\beta,n,m)}e_{\gamma}(y,x,xy,1)\hbar^{m},$$
where the vectors $\gamma=(\gamma_{100},\gamma_{010},\gamma_{110},\gamma_{111})\in\mathbb{N}^4$ are such that $|\gamma|\leq n$ and
$$\gamma_{100}+\gamma_{110}+\gamma_{111}=\alpha, \ \ \ \ \ \ \gamma_{010}+\gamma_{110}+\gamma_{111}=\beta, \ \ \ \ \ \ \mbox{and}
\ \ \ \ \ \ \gamma_{111}=m.$$
For example, for $n=3, \ \alpha=2, \ \beta=3$, we have
$$e_2(y)\star e_3(x)\ = \ e_{(1,2)}(x,xy)\ + \  e_{(1,1,1)}(x,xy,1)\hbar+e_{(1,2)}(x,1)\hbar^2$$
since in this case $\gamma=(\gamma_{100}, \gamma_{010}, \gamma_{110},\gamma_{111})\in\mathbb{N}^4$ is such that
$$\gamma_{100}+\gamma_{010}+ \gamma_{110}+\gamma_{111}\leq 3, \ \ \ \ \gamma_{100}+\gamma_{110}+\gamma_{111}=2, \ \ \ \
\gamma_{010}+ \gamma_{110}+\gamma_{111}=3, \ \ \ \ \gamma_{111}=m.$$
Solving this equation for $m=0,1,2$ we, respectively, obtain $$ \gamma=(0,1,2,0),\ \ \ \ \ \gamma=(0,1,1,1),  \ \ \ \ \ \mbox{and} \ \ \ \ \ \gamma=(0,1,0,2),$$ yielding the desired result.\\

On the other hand,  from  Definition \ref{vacdef} we get that
$$e_{2}(y)=y_1y_2+y_1y_3+y_2y_3, \ \  e_3(x)=x_1x_2x_3, \ \ \mbox{and thus}$$
$$e_{2}(y)\star e_3(x) \ = \ (y_1y_2+y_1y_3+y_2y_3)\star(x_1x_2x_3)\ =$$ $$(x_1x_2x_3y_1y_2 + x_1x_2x_3y_1y_3 + x_1x_2x_3y_2y_3) \ + $$
$$ (x_1x_3y_1+x_1x_2y_1+x_2x_3y_2+x_1x_2y_2+x_2x_3y_3+x_1x_3y_3)\hbar\ + \ (x_1+x_2+x_3)\hbar^2,$$
which indeed is equal to
 $$e_{(1,2)}(x,xy)\ + \  e_{(1,1,1)}(x,xy,1)\hbar\ + \ e_{(1,2)}(x,1)\hbar^2.$$

Note that $\ \ \{e_2(y),e_3(x)\}=2 e_{(1,1,1)}(x,xy,1).$}
\end{exa}

\begin{exa}{\em Let $p=q=xy\in \mathbb{R}[x,y], \ $ then we have that
$$e_{\alpha}(xy)\star e_{\beta}(xy)\ = \ \sum_{m\geq 0}\sum_{\gamma\in \mathrm{QL}(\alpha,\beta,n,m)}e_{\gamma}(xy,xy,x^2y^2,xy)\hbar^m,$$
where  the vector $\ \gamma=(\gamma_{100},\gamma_{010},\gamma_{110},\gamma_{111})\in\mathbb{N}^4\ $ is such that $|\gamma|\leq n$ and
$$\gamma_{100}+\gamma_{110}+\gamma_{111}=\alpha, \ \ \ \ \ \gamma_{010}+\gamma_{110}+\gamma_{111}=\beta  \ \ \ \ \ \mbox{and}
\ \ \ \ \ \gamma_{111}=m.$$
Thus for $n=2,\ \alpha=2, \ \beta=1$,  we get
$$e_2(xy)\star e_1(xy)\ = \ e_{(1,1)}(xy,x^2y^2)\ + \  e_{(1,1)}(xy,xy)\hbar$$
since in this case $\ \gamma=(\gamma_{100}, \gamma_{010}, \gamma_{110},\gamma_{111})\in\mathbb{N}^4\ $ is such that
$$\gamma_{100}+\gamma_{010}+ \gamma_{110}+\gamma_{111}\leq 2, \ \ \ \ \ \gamma_{100}+\gamma_{110}+\gamma_{111}=2, \ \ \ \ \
\gamma_{010}+ \gamma_{110}+\gamma_{111}=1, \ \ \ \ \ \gamma_{111}=m.$$
Solving this equation for $m=0,1$ we, respectively, obtain
$$\gamma=(1,0,1,0) \ \  \  \mbox{and} \ \ \ \gamma=(1,0,0,1), \ \ \  \mbox{yielding the desired result.}$$

From  Definition \ref{vacdef} we have
$e_{2}(xy)=x_1y_1x_2y_2, \ \  e_1(xy)=x_1y_1+x_2y_2, \ $ and thus
$$e_{2}(xy)\star e_1(xy) = (x_1y_1x_2y_2) \star(x_1y_1+x_2y_2)=(x_1^2y_1^2x_2y_2+x_1y_1x_2^2y_2^2)\ +\ 2(x_1y_1x_2y_2)\hbar,$$
which indeed is equal to
$e_{(1,1)}(xy,x^2y^2)\ + \ e_{(1,1)}(xy,xy)\hbar.$\\

Note that $\ \ \{e_2(xy),e_1(xy)\}=2e_{(1,1)}(xy,xy).$
}
\end{exa}


\begin{exa}\em{
Let $n=2, \ \alpha=\beta=2$, then
$$e_2(x^ay)\star e_2(xy^b)\ = \ e_2(x^{a+1}y^{b+1})\ + \ e_{(1,1)}(x^{a+1}y^{b+1},x^ay^b)\hbar \ + \ e_2(x^ay^b)\hbar^2.$$
Using Theorem \ref{prin} we have that $$e_2(x^ay)\star e_2(xy^b)\ = \ \sum_{\gamma\in \mathrm{QL}(2,2,2,m)}e_{\gamma}(x^ay,xy^b,x^{a+1}y^{b+1},x^ay^b)\hbar^m,$$ where  $\ \gamma=(\gamma_{100}, \gamma_{010}, \gamma_{110},\gamma_{111})\in\mathbb{N}^4$ is such that
$$\gamma_{100}+\gamma_{010}+ \gamma_{110}+\gamma_{111}\leq 2, \ \
\gamma_{100}+\gamma_{110}+\gamma_{111}=2, \ \
\gamma_{010}+ \gamma_{110}+\gamma_{111}=2, \ \
\gamma_{111}=m.$$
Solving this equation for $\ m=0,1,2\ $ we, respectively, obtain
$$\gamma=(0,0,2,0), \ \ \ \ \gamma=(0,0,1,1) \ \ \ \ \mbox{and} \ \ \ \ \gamma=(0,0,0,2).$$
Thus we get that
$$e_2(x^ay)\star e_2(xy^b)\ = \ e_2(x^{a+1}y^{b+1})\ + \ e_{(1,1)}(x^{a+1}y^{b+1},x^ay^b)\hbar\ + \ e_2(x^ay^b)\hbar^2.$$
On the other hand  from Definition \ref{vacdef} we have that
$$e_{2}(x^ay)=x_1^ay_1x_2^ay_2 \ \ \ \ \mbox{and} \ \ \ \ e_2(xy^b)=x_1y_1^bx_2y_2^b.$$ Computing directly the $\star$-product we obtain
$$e_{2}(x^ay)\star e_2(xy^b) \ = \ (x_1^ay_1x_2^ay_2) \star (x_1y_1^bx_2y_2^b) \ =$$ $$x_1^{a+1}y_1^{b+1}x_2^{a+1}y_2^{b+1}\ + \ (x_1^{a}y_1^{b}x_2^{a+1}y_2^{b+1}+x_1^{a+1}y_1^{b+1}x_2^{a}y_2^{b})
\hbar\ + \ x_1^{a}y_1^{b}x_2^{a}y_2^{b}\hbar^2\ =$$
$$e_2(x^{a+1}y^{b+1})\ + \ e_{(1,1)}(x^{a+1}y^{b+1},x^ay^b)\hbar\ + \ e_2(x^ay^b)\hbar^2.$$

Note that $\ \ \{e_{2}(x^ay),e_2(xy^b)\}=2e_{(1,1)}(x^{a+1}y^{b+1},x^ay^b).$
}
\end{exa}

\

We close this work stating the main problem that our research opens.

\begin{pro}\em{ Describe the relations in the algebra of quantum symmetric functions.
}
\end{pro}

\subsection*{Acknowledgment}
We thank Camilo Ortiz and Fernando Novoa  for helping us to develop the software used in the examples.
Eddy Pariguan thanks "Pontificia Universidad Javeriana" for providing funds for this research through the project  ID-00006185 titled
"Descripci\'on de producto cu\'antico para funciones multi-sim\'etricas."

\

\noindent ragadiaz@gmail.com\\
\noindent Instituto de Matem\'aticas y sus Aplicaciones, Universidad Sergio Arboleda, Bogot\'a, Colombia\\

\noindent epariguan@javeriana.edu.co\\
\noindent Departamento de Matem\'aticas, Pontificia Universidad Javeriana, Bogot\'a, Colombia\\

\end{document}